\documentclass[11pt]{amsart}

\usepackage[latin1]{inputenc} 
\usepackage{amsthm,amsmath,amssymb,bm,mathtools} 
\usepackage{graphicx}
\usepackage{color}
\usepackage{verbatim}
\usepackage{float}
\usepackage{bbm}
\usepackage{enumerate}
\usepackage{hyperref}

\usepackage{array}   
\newcolumntype{C}{>{$}c<{$}}
\newcolumntype{L}{>{$}l<{$}}

\theoremstyle{plain}
\newtheorem{theo}{Theorem}[section]

\newtheorem{lemma}[theo]{Lemma}

\newtheorem{conj}[theo]{Conjecture}
\newtheorem{claim}[theo]{Claim}

\theoremstyle{definition}
\newtheorem{defn}[theo]{Definition}

\newcommand{\ra}{\rightarrow}

\newenvironment{proofclaim}{\removelastskip\penalty55\medskip\noindent{\it Proof of the claim. }}

\newcommand{\qedclaim}{\hfill\scalebox{0.6}{$\blacksquare$}\\}

\newcommand\floor[1]{\left\lfloor#1\right\rfloor}
\newcommand\ceil[1]{\left\lceil#1\right\rceil}

\title{A note on improved bounds for hypergraph rainbow matching problems}
\author{Candida Bowtell, Andrea Freschi, Gal Kronenberg, Jun Yan}
\thanks{CB: School of Mathematics, University of Birmingham, Edgbaston, Birmingham, United Kingdom, \href{mailto:c.bowtell@bham.ac.uk}{\tt c.bowtell@bham.ac.uk}, supported by Leverhulme Trust Early Career Fellowship ECF--2023--393.\\
AF: HUN-REN, Alfr\' ed R\'enyi Institute of Mathematics, Budapest, Hungary, \href{mailto:freschi.andrea@renyi.hu}{\tt freschi.andrea@renyi.hu}, research partially supported by ERC Advanced Grants ``GeoScape'', no. 882971 and ``ERMiD'', no. 101054936.\\
GK: Mathematical Institute, University of Oxford, United Kingdom,
\href{mailto:kronenberg@maths.ox.ac.uk}{\tt kronenberg@maths.ox.ac.uk}, supported by the Royal Commission for
the Exhibition of 1851.\\ JY: Mathematics Institute, University of Warwick, Coventry, United Kingdom, \href{mailto:jun.yan@warwick.ac.uk}{\tt jun.yan@warwick.ac.uk}, supported by the Warwick Mathematics Institute CDT, and by funding from the UK EPSRC (Grant number: EP/W523793/1)}

\begin{document}

\date{}

\begin{abstract} 
A natural question, inspired by the famous Ryser-Brualdi-Stein Conjecture, is to determine the largest positive integer $g(r,n)$ such that every collection of $n$ matchings, each of size~$n$, in an $r$-partite $r$-uniform hypergraph contains a {\it rainbow matching} of size~$g(r,n)$.
The parameter $g'(r,n)$ is defined identically with the exception that the host hypergraph is not required to be $r$-partite.

In this note, we improve the best known lower bounds on $g'(r,n)$ for all $r \geq 4$ and the upper bounds on $g(r,n)$ for all $r \geq 3$, provided $n$ is sufficiently large.
More precisely, we show that if $r\ge3$ then
$$\frac{2n}{r+1}-\Theta_r(1)\le g'(r,n)\le g(r,n)\le n-\Theta_r(n^{1-\frac{1}{r}}).$$ 
Interestingly, while it has been conjectured that $g(2,n)=g'(2,n)=n-1$, our results show that if~$r\ge3$ then $g(r,n)$ and $g'(r,n)$ are bounded away from $n$ by a function which grows in~$n$.

We also prove analogous bounds for the related problem where we are interested in the smallest size $s$ for which any collection of $n$ matchings of size $s$ in an ($r$-partite) $r$-uniform hypergraph contains a rainbow matching of size $n$.
\end{abstract}

\maketitle

\section{Introduction}

Let $M_1,\dots,M_n$ be a collection of (not necessarily disjoint) matchings in a hypergraph. 
A matching $M\subset \bigcup_{i=1}^n M_i$ is called \textit{rainbow} if there exists an injection $\phi:M\to[n]$ such that for every edge $e$ in $M$, $e$ is also an edge in $M_{\phi(e)}$. 
Equivalently, one can view $[n]$ as a list of~$n$ colours and the edges in $M_i$ as being colourable with colour~$i$.
In this language, a matching $M\subset \bigcup_{i=1}^n M_i$ is rainbow if its edges can be coloured using different colours.

The famous Ryser-Brualdi-Stein Conjecture claims that 
every decomposition of the edges of the complete bipartite graph $K_{n,n}$ into $n$ edge-disjoint matchings of size $n$, $M_1, \ldots, M_n$,
admits a rainbow matching $M \subset \bigcup_{i \in [n]} M_i$ of size $n-1$ when $n$ is even, and a rainbow matching of size $n$ when $n$ is odd. 
This conjecture has received a huge amount of attention with the best bound being a recent breakthrough of Montgomery~\cite{mont}, who proved that one can always find a rainbow matching of size at least $n-1$ provided $n$ is sufficiently large. 
Note that the $n-1$ term in the even case of the Ryser-Brualdi-Stein Conjecture is best possible as shown by examples coming from Cayley tables of certain abelian groups such as $\mathbb{Z}_{n}$, where $n$ is even.
 
The study of generalisations of this problem has also gained significant traction in recent years. 
The new results of this note are concerned with some generalisations of the problem to hypergraphs, but it is helpful to first discuss the state of the art in the graph setting. 
In \cite{ab}, Aharoni and Berger conjectured the following generalisation of the Ryser-Brualdi-Stein Conjecture.

\begin{conj}\label{conj:ab_bip}
Every collection of $n$ matchings of size $n$ in a bipartite graph admits a rainbow matching of size $n-1$.
\end{conj}

Crucially, unlike the Ryser-Brualdi-Stein Conjecture, the $n$ matchings in Conjecture~\ref{conj:ab_bip} do not need to be edge-disjoint, nor do they all necessarily cover the same set of vertices.  
Conjecture~\ref{conj:ab_bip} was further extended by dropping the requirement that the host graph is bipartite.
This version of the conjecture, still attributed to Aharoni and Berger, was first stated in~\cite{abcz}.

\begin{conj}\label{conj:ab}
Every collection of $n$ matchings of size $n$ in a graph admits a rainbow matching of size $n-1$.
\end{conj}

The term~$n-1$ in Conjectures~\ref{conj:ab_bip} and~\ref{conj:ab} is best possible for all $n$, as shown by the following example. 
Let $M$ and $M'$ be two edge-disjoint matchings of size $n$ which form a cycle $C_{2n}$ on $2n$ vertices.
Let $M_i:=M$ for every $1\le i\le n-1$, and let $M_n:=M'$.
Then $\{M_1,\dots,M_n\}$ is a collection of $n$ matchings of size $n$ in a bipartite graph and there is no rainbow matching of size $n$.
Indeed, any matching of size $n$ in $\bigcup_{i=1}^n M_i = C_{2n}$ is either a copy of $M$ or $M'$, and so its edges cannot receive more than $n-1$ different colours.

Another natural line of inquiry is to determine the smallest positive integer $s$ such that any collection of $n$ matchings of size $s$ admits a rainbow matching of size $n$.
In this direction, Aharoni and Berger \cite{ab} conjectured the following.

\begin{conj}\label{conj:ab_bip_strong}
Every collection of $n$ matchings of size $n+1$ in a bipartite graph admits a rainbow matching of size $n$.
\end{conj}

Note that Conjecture \ref{conj:ab_bip_strong} implies Conjecture \ref{conj:ab_bip}: given $n$ matchings of size $n$ in a bipartite graph~$G$, one introduces a dummy edge $e$ disjoint from $G$ and adds it to each matching.
This yields $n$ matchings of size $n+1$ in the bipartite graph $G\cup e$.
Applying Conjecture \ref{conj:ab_bip_strong} and removing the dummy edge $e$ implies Conjecture \ref{conj:ab_bip}.
In particular, it follows that the term $n+1$ in Conjecture \ref{conj:ab_bip_strong} is best possible.

For the non-bipartite case, the following conjecture was implicitly made in \cite{abchs} and stated in \cite{grww}.

\begin{conj}\label{conj:ab_strong}
Every collection of $n$ matchings of size $n+2$ in a graph admits a rainbow matching of size $n$.
\end{conj}

To see why the term $n+2$ in Conjecture~\ref{conj:ab_strong} is best possible, consider the following construction for~$n$ odd. 
Take the disjoint union of $(n+1)/2$ copies of a properly edge-coloured~$K_4$ with colours red, green and blue. 
Let $M_1, \ldots, M_{n-2}$ be copies of the matching formed by the red edges.
Let $M_{n-1}$ and $M_n$ be copies of the green and blue matchings respectively. 
Then each $M_i$ is a matching of size $n+1$, but $M_1, \ldots, M_n$ does not admit a rainbow matching of size $n$.

Conjectures \ref{conj:ab_bip}--\ref{conj:ab_strong} have now all been settled asymptotically.
More precisely, it is known that any collection of $n$ matchings of size $n$ in a graph yields a rainbow matching of size $n-o(n)$ (see \cite{bgs,mcps}) and, similarly, any collection of $n$ matchings of size $n+o(n)$ yields a rainbow matching of size $n$ (see \cite{mcps,p}).
We refer to the former type of statement as a {\it weak} asymptotic and to the latter as a {\it strong} asymptotic, following the same terminology in~\cite{mcps}.
The current best-known values of the $o(n)$-terms are shown in Table~\ref{table:1}.

The choice of the terms {\it weak} and {\it strong} was motivated by the fact that, using a trivial reduction argument via dummy edges, a strong asymptotic statement immediately implies a weak asymptotic statement.
Perhaps surprisingly, Munh\'{a}-Correia, Pokrovskiy and Sudakov~\cite{mcps} were able to show, using a simple but powerful sampling trick, that, in these settings, the weak asymptotic statement, in fact, also implies the strong asymptotic statement.

\begin{table}[ht]
\caption{}\label{table:1}
\centering
\begin{tabular}{ |c|c|c|c|c| }
 \hline
 {\bf Host graph} & {\bf \# of matchings} & {\bf Size of each matching} & {\bf Size of rainbow matching} & {\bf Ref.} \\ 
 \hline
 Bipartite & $n$ & $n$ & $n-\sqrt{n}$ & \cite{bgs} \\ 
  \hline
 Any & $n$ & $n$ & $n-20n^{7/8}$ & \cite{mcps} \\ 
  \hline
 Bipartite & $n$ & $n+7n^{3/4}$ & $n$ & \cite{mcps} \\ 
  \hline
 Any & $n$ & $n+20n^{15/16}$ & $n$ & \cite{mcps} \\ 
 \hline
\end{tabular}
\end{table}   

We now shift our attention to the hypergraphs setting.
A hypergraph is {\it $r$-uniform} if every edge contains exactly $r$ vertices. 
An $r$-uniform hypergraph $H$ is {\it $r$-partite} if $V(H)$ can be partitioned into~$r$ sets $V_1\cup\cdots\cup V_r$, such that every edge in $H$ contains exactly one vertex from each $V_i$.

Aharoni, Charbit and Howard~\cite{ach} considered the following generalisation of Conjecture~\ref{conj:ab_bip}: given any $n$ matchings of size $n$ in an $r$-partite $r$-uniform hypergraph, what is the largest rainbow matching one can guarantee to find?
In this note, we improve the best known bounds for this question, and the related question where the $r$-partite restriction is dropped.
We also prove corresponding bounds for the analogous generalisation of Conjectures~\ref{conj:ab_bip_strong}
and~\ref{conj:ab_strong}, which, to the best of our knowledge, have not been considered previously in the literature: what is the smallest $s$ such that any $n$ matchings of size $s$ in an ($r$-partite) $r$-uniform hypergraph contain a rainbow matching of size $n$?

We start by discussing the first problem. We use the following notation, matching and extending that from~\cite{ach}. 

\begin{defn}[$g(r,n)$, $g'(r,n)$]
\mbox{}

\begin{itemize}
    \item Let $g(r,n)$ be the largest $s$ such that every collection of $n$ matchings of size $n$ in an $r$-partite $r$-uniform hypergraph admits a rainbow matching of size $s$. 
    \item Let~$g'(r,n)$ be the largest $s$ such that every collection of $n$ matchings of size $n$ in an $r$-uniform hypergraph admits a rainbow matching of size $s$.
\end{itemize} 
\end{defn}

Note trivially that $g'(r,n) \leq g(r,n)$ for all $r$ and $n$. Aharoni, Charbit and Howard~\cite{ach} focused only on the $r$-partite case and showed that for all $r \geq 3$, $g(r,n) \leq n-2^{r-2}$ when $n$ is even and $g(r,n) \leq n-2^{r-2}+1$ when $n$ is odd, whilst $g(3,n) \geq n/2$ for all values of $n$. Note that a greedy argument allows us to show that $g'(r,n) \geq \lceil{n/r}\rceil$. More recently, Aharoni, Berger, Chudnovsky, Zerbib~\cite{abcz} improved this lower bound to $g'(r,n) \geq \frac{n}{r-1/2}+O(1/r)$ for all $r$. Note that this also gives improved lower bounds for $g(r,n)$ for all $r \geq 4$, though the bound of Aharoni, Charbit and Howard~\cite{ach} remains the best lower bound when $r=3$ in the partite setting. 

The main purpose of this note is two-fold. Firstly, we give a short proof that $g'(r,n) \geq \frac{2n}{r+1}-f(r)$, for every $r \geq 3$, where $f(r)$ is a function depending only on $r$ and not $n$. More precisely, we show the following.

\begin{theo}\label{theo:lower}
   For every $r \geq 3$, 
   $$g'(r,n)\geq \frac{2n}{r+1}-\frac{\binom{2r}r}{r+1}.$$
\end{theo}

This improves the results of Aharoni, Berger, Chudnovsky, Zerbib~\cite{abcz} when $n$ is sufficiently large relative to $r$. The proof is very short and comes from generalising and slightly simplifying the arguments in \cite{ach} which show that $g(3,n) \geq n/2$, as well as using a well-known result of Bollob\'{a}s~\cite{bollobas} concerning cross-intersecting families. However, this result seems far from tight, and we suspect that in fact $g(r,n)=n-o(n)$ for all $r$, but that a significantly more complicated strategy would be required to achieve this bound. 

In view of this, perhaps our second result, regarding upper bounds on $g(r,n)$ yields a more interesting observation. Whilst Aharoni, Charbit and Howard~\cite{ach} did not go as far as to conjecture that their construction showing that $g(r,n) \leq n-2^{r-2}$ when $n$ is even is the best possible upper bound construction, they observed that if $g(r,n) \geq n-2^{r-2}$, this would naturally generalise Conjecture \ref{conj:ab_bip} from the case~$r=2$. Our next result shows that the behaviour in the case $r\geq 3$ does not actually fit with such a conjecture. 

\begin{theo}\label{theo:upper}
For every $r\ge3$, and $n>6^r$, $$g(r,n)\leq n-\frac1{12r}n^{\frac{r-1}r}.$$
\end{theo}

That is, for sufficiently large $n$ in terms of $r\ge3$, there exists a collection of $n$ matchings of size $n$ in an $r$-partite $r$-uniform hypergraph with no rainbow matching of size larger than $n-\frac1{12r}n^{\frac{r-1}r}$. 
In particular, the size of a largest rainbow matching is bounded away from $n$ not just by a function of $r$, but by a function that grows with $n$. If Conjecture \ref{conj:ab_bip} is true, then this shows quite different behaviour between the case where $r=2$ and the cases where $r \geq 3$. 

Whilst we do not suggest that Conjecture \ref{conj:ab_bip} is not true, it is perhaps interesting to note that if this upper bound generalised to the case $r=2$, it would match closely with the current best known lower bound $g(2,n)\ge n-\sqrt{n}$ (see Table \ref{table:1}).

Our proof of Theorem \ref{theo:upper} is again short, though crucially relies on a result of Pohoata, Sauermann and Zakharov~\cite{psz}, which was used to make significant progress on a related problem. We say more about this in Section \ref{sec_ub} and the concluding remarks.

As previously stated, we also introduce the analogous generalisations of Conjectures \ref{conj:ab_bip_strong} and \ref{conj:ab_strong} to hypergraphs, and prove related bounds. 

\begin{defn}[$h(r,n)$, $h'(r,n)$]
\mbox{}
\begin{itemize}
    \item Let $h(r,n)$ be the smallest $s$ such that every collection of $n$ matchings of size $s$ in an $r$-partite $r$-uniform hypergraph admits a rainbow matching of size $n$.
    \item Let $h'(r,n)$ be the smallest $s$ such that every collection of $n$ matchings of size $s$ in an $r$-uniform hypergraph admits a rainbow matching of size $n$. 
\end{itemize}
\end{defn}

It is clear that $h'(r,n) \geq h(r,n)$.  Moreover, we prove the following.

\begin{theo}\label{theo:strong}
For every $r \geq 3$ and $n$ sufficiently large, we have that
    $$n+\frac1{12r}n^{\frac{r-1}{r}} \leq h(r,n)\leq h'(r,n) \leq \frac{(r+1)n}{2}+3r^2n^{\frac{2r-1}{2r}}.$$
\end{theo}

The lower bound in Theorem \ref{theo:strong} follows easily from Theorem \ref{theo:upper} by adding dummy edges, while the upper bound proof combines the ideas for proving Theorem \ref{theo:lower} with the aforementioned sampling trick introduced by Munh\'{a}-Correia, Pokrovskiy and Sudakov. 

\subsection{Organisation} The rest of the paper is set out as follows.
The proofs of our results are found in Section \ref{sec:proofs}.
In Section \ref{sec_conc} we give some concluding remarks including a brief discussion of some related literature that feels amiss to leave out of this note.

\section{Proofs}\label{sec:proofs}

In this section we prove Theorems \ref{theo:lower}, \ref{theo:upper} and \ref{theo:strong}.

\subsection{Proof of Theorem \ref{theo:lower}} \label{sec_lb}

We deduce Theorem \ref{theo:lower} from the following stronger statement.

\begin{lemma}\label{lem:gibounds}
Let $\mathcal{H}=\{H_1, \ldots, H_n\}$ be a collection of $n$ matchings of size $N$ in an $r$-uniform hypergraph. 
Let $m$ be the size of a maximum rainbow matching for $\mathcal{H}$.
Then \[(n-m)\frac{2N-(r+1)m}{r-1}\leq\frac12\binom{2r}rm.\]
\end{lemma}

We first show how Lemma \ref{lem:gibounds} implies Theorem \ref{theo:lower}.

\begin{proof}[Proof of Theorem \ref{theo:lower}]
Let $\mathcal{H}=\{H_1, \ldots, H_n\}$ be a collection of $n$ matchings of size $n$ in an $r$-uniform hypergraph. 
Let $m$ be the size of a maximum rainbow matching for $\mathcal{H}$. 
By Lemma \ref{lem:gibounds}, 
\[(n-m)\frac{2n-(r+1)m}{r-1}\leq \frac12\binom{2r}{r}m.\]
After rearranging, we get
\[(r+1)m^2-\left((r+3)n+\frac12(r-1)\binom{2r}{r}\right)m+2n^2\leq 0.\]
Solving this gives
\begin{align*}
m&\geq\frac{(r+3)n+\frac12(r-1)\binom{2r}{r}-\sqrt{\left((r+3)n+\frac12(r-1)\binom{2r}{r}\right)^2-8(r+1)n^2}}{2(r+1)}\\
&=\frac{(r+3)n+\frac12(r-1)\binom{2r}{r}-\sqrt{(r-1)^2n^2+(r-1)(r+3)\binom{2r}{r}n+\frac14(r-1)^2\binom{2r}{r}^2}}{2(r+1)}\\
&=\frac{(r+3)n+\frac12(r-1)\binom{2r}{r}-\sqrt{\left((r-1)n+\frac12(r+3)\binom{2r}r\right)^2-2(r+1)\binom{2r}2^2}}{2(r+1)}\\
&\geq\frac{(r+3)n+\frac12(r-1)\binom{2r}{r}-\sqrt{\left((r-1)n+\frac12(r+3)\binom{2r}r\right)^2}}{2(r+1)}\\
&=\frac{(r+3)n+\frac12(r-1)\binom{2r}{r}-(r-1)n-\frac12(r+3)\binom{2r}r}{2(r+1)}=\frac{2n}{r+1}-\frac{\binom{2r}r}{r+1},
\end{align*}
completing the proof.
\end{proof}

To prove Lemma \ref{lem:gibounds}, we follow closely the double-counting strategy of \cite[Theorem 1.13]{ach} used to show that $g(3,n) \geq n/2$. This involves proving two short claims. 
For the first claim, we generalise and simplify slightly the proof given by \cite{ach}. 
For the second claim, we use the well-known `two families theorem' of Bollob\'{a}s~\cite{bollobas} concerning cross-intersecting set pair systems. 
A {\it cross-intersecting set pair system} of {\it size} $m \geq 2$ consists of two collections of finite sets $\mathcal{A}=\{A_1, \ldots, A_m\}$ and $\mathcal{B}=\{B_1, \ldots, B_m\}$ each of size $m$, such that $A_i \cap B_i = \emptyset$ for every $i \in [m]$, and $A_i \cap B_j \neq \emptyset$ for every $1 \leq i \neq j \leq m$.

\begin{theo}[Bollob\'{a}s~\cite{bollobas}] \label{theo:bollobas}
Let $(\mathcal{A}, \mathcal{B})$ be a cross-intersecting set pair system of size $m \geq 2$. Then 
$$\sum_{i \in [m]} \binom{|A_i|+|B_i|}{|A_i|}^{-1} \leq 1.$$
\end{theo}

We are now ready to prove Lemma \ref{lem:gibounds}.

\begin{proof}[Proof of Lemma \ref{lem:gibounds}]
Let $\mathcal{H}=\{H_1, \ldots, H_n\}$ be a collection of $n$ matchings of size $N$ in an $r$-uniform hypergraph. 
Let $M$ be a maximum rainbow matching for $\mathcal{H}$ and let $m:=|M|$.
We say that the edges in $H_i$ have colour $i$ for every $i\in[n]$. 
Reordering if necessary, we may assume that the colours used in $M$ are $1, \ldots, m$. 

We say that an edge $e \in M$ is {\it good} for colour $i>m$ if there exist edges $f_1, f_2 \in E(H_i)$ such that $f_j \cap e = f_j \cap V(M)$ for $j \in \{1,2\}$. Let $g_i$ denote the number of good edges for colour $i>m$.

\begin{claim}\label{claim:goodlower}
    For each $i > m$ we have that $g_i \geq \frac{2N-(r+1)m}{r-1}$.
\end{claim}
\begin{proofclaim}
    For every $i>m$, by the maximality of $M$, every edge in $H_i$ must intersect some edge in $M$.
    Let $H_i'$ be the set of edges in $H_i$ that intersect exactly one edge in $M$, and let $h_i=|H_i'|$. 
    Then, each edge in $H_i\setminus H_i'$ must intersect $M$ in at least two vertices.
    Since $H_i$ is a matching, edges in $H_i$ intersect edges in $M$ in distinct vertices, and so we have $|H_i'|+2|H_i\setminus H_i'|\le r|M|$. 
    It follows that $h_i+2(N-h_i)\leq rm$, and thus $h_i\geq 2N-rm$. From the definition, every edge in $M$ that is not good for colour $i$ intersects at most one edge in $H_i'$, so $rg_i+(m-g_i)\geq h_i$. Combining these two inequalities yields $(r-1)g_i\geq 2N-(r+1)m$, and thus $g_i\geq\frac{2N-(r+1)m}{r-1}$.\qedclaim
\end{proofclaim}
\begin{claim}\label{claim:goodupper}
   Any edge $e \in M$ is good for at most $\frac{1}{2}\binom{2r}{r}$ colours $i>m$.
\end{claim}
\begin{proofclaim}
    Suppose that some edge $e\in M$ is good for $\ell$ colours $c_1,\ldots,c_\ell>m$, then from the definition, for every $i\in[\ell]$ there exist edges $f_i,f_i'$ of colour $c_i$, such that $f_i\cap e=f_i\cap V(M)$ and $f_i'\cap e=f_i'\cap V(M)$. If there exist distinct $i,j\in [\ell]$ such that one of $f_i,f_i'$, say $f_i$, and one of $f_j,f_j'$, say $f_j$, are disjoint, then $M\setminus\{e\} \cup \{f_i, f_j\}$ is a rainbow matching in $\mathcal{H}$ of size $|M|+1$, contradicting the maximality of $M$. 
    Hence, for all distinct $i,j\in[\ell]$, all of $f_i\cap f_j, f_i\cap f_j', f_i'\cap f_j, f_i'\cap f_j'$ are non-empty. This means that $(\{f_1,\ldots, f_\ell,f_1',\ldots, f_\ell'\},\{f_1',\ldots, f_\ell',f_1,\ldots, f_\ell\})$ form a cross-intersecting set pair system of size $2\ell$. By Theorem \ref{theo:bollobas}, since $|f_i|=|f_i'|=r$ for all $i$, we have $2\ell/\binom{2r}{r}\leq1$, and so $\ell\le\frac12\binom{2r}{r}$.\qedclaim
\end{proofclaim}
Combining these two claims, we get 
\[(n-m)\frac{2N-(r+1)m}{r-1}\leq \sum_{i>m} g_i \leq \frac12\binom{2r}{r}m,\]
as required.
\end{proof}

\subsection{Proof of Theorem \ref{theo:upper}} \label{sec_ub}

Before proving Theorem \ref{theo:upper}, we return to the previous best upper bound on $g(r,n)$ observed by Aharoni, Charbit and Howard~\cite{ach}. Since it is not explicitly written out in \cite{ach}, we include here a brief description of their construction that shows $g(r,n) \leq n-2^{r-2}$ when $n$ is even. Let $G$ be a hypergraph with vertex set $V=\{a_1, \ldots, a_r, b_1, \ldots, b_r\}$ and the following set of coloured edges. For each of the $2^{r-1}$ subsets $T$ of $[r]$ containing the element $1$, let $e_T=\{a_i:i \in [T]\} \cup \{b_i: i \in [r]\setminus T\}$, $f_T=V(G) \setminus e_T$, and include both of these edges in $G$, coloured with colour $\gamma_T$. Then $G$ is an $r$-uniform $r$-partite hypergraph with $2^r$ edges coloured with $2^{r-1}$ colours, such that each colour class contains 2 edges that form a perfect matching of $G$. Moreover, note that the largest rainbow matching in $G$ has size $1$. Relabel the $2^{r-1}$ colours as $c_1,\ldots,c_{2^{r-1}}$. Let $H$ be the union of $n/2$ disjoint copies of $G$. For $i \in [2^{r-1}-1]$, let $M_i$ be the matching consisting of all $n$ edges of colour $c_i$ in $H$. For $i \geq 2^{r-1}$, let $M_i$ be the matching consisting of all $n$ edges of colour $c_{2^{r-1}}$ in $H$. Note that for any $j \in [2^{r-1}-1]$, whenever we choose an edge in the matching $M_j$, this intersects exactly two edges in every $M_i$ with $i \neq j$, and for any $i \geq 2^{r-1}$, any edge in $M_i$ intersects exactly two edges in every $M_j$ such that $j \in [2^{r-1}-1]$. Thus our best strategy is to choose $n-2^{r-1}$ disjoint edges from $M_{2^{r-1}+1}, \ldots, M_n$, and find a rainbow matching of size $2^{r-2}$ in $M_1, \ldots, M_{2^{r-1}}$, giving a rainbow matching of size at most $n-2^{r-1}+2^{r-2}=n-2^{r-2}$.

To prove Theorem \ref{theo:upper}, we use the following result of Pohoata, Sauermann and Zakharov~\cite[Theorem 3.1]{psz}, which arises from considering a somewhat more complex construction. For ease of application in our setting we make a small change, but our statement follows easily from the original by noting that $\floor{(t/3r)^r} \leq (\floor{t/r}-1)^r$ when $t \geq 3r$. 

\begin{theo}[\cite{psz}]\label{theo:psz}
Let $r\geq3$, $t\geq 3r$ and $n=\floor{(t/3r)^r}$ be integers. 
Then there exists a collection of $n$ matchings of size $t$ in an $r$-uniform $r$-partite hypergraph with $tr$ vertices that does not admit a rainbow matching of size $t$.
\end{theo}

Our proof of Theorem \ref{theo:upper} also uses the following lemma.

\begin{lemma}\label{lemma:blowup}
Suppose for each $i\in[q]$ there exists a collection of $n$ matchings of size $t_i$ in an $r$-partite $r$-uniform hypergraph $H_i$ that does not admit a rainbow matching of size $t_i$. Then, there exists a collection of $n$ matchings of size $\sum_{i=1}^qt_i$ in an $r$-uniform $r$-partite hypergraph $H$ that does not contain a rainbow matching of size $\sum_{i=1}^qt_i-q+1$.
\end{lemma}
\begin{proof}
Let $H$ be the $r$-partite $r$-uniform hypergraph formed by taking the disjoint union of the hypergraphs $H_1,\ldots,H_q$. Define $n$ matchings $M_1,\ldots,M_n$ of size $n$ on $H$ so that for each $i\in[q]$, the restriction of $M_1,\ldots,M_n$ on $H_i$ is exactly the collection of $n$ matchings of size $t_i$ on $H_i$ that does not have a rainbow matching of size $t_i$. Then, the maximum size of a rainbow matching for $M_1,\ldots,M_n$ is at most $\sum_{i=1}^q(t_i-1)=\sum_{i=1}^qt_i-q$.
\end{proof}

Combining these two results allows us to prove Theorem \ref{theo:upper}.

\begin{proof}[Proof of Theorem \ref{theo:upper}]
For every $n>6^r$, there exists a unique $a\geq6$ such that $a^r<n\leq(a+1)^r$. Set $t=3(a+1)r$. Let $n=tq+s$, where $0\leq s<t$. Note that $q=\floor{\frac nt}=\floor{\frac n{3(a+1)r}}\geq\frac n{6(a+1)r}\geq\frac1{12r}n^{\frac{r-1}r}$, as $\frac n{3(a+1)r}>\frac{a^r}{3(a+1)r}>\frac{a^{r-1}}{6r}\geq2$ and $n^{\frac1r}>a>\frac{a+1}2$. 

Apply Theorem \ref{theo:psz} with $t=3(a+1)r$ and using $n\leq(a+1)^r=\floor{(t/3r)^r}$, we obtain a collection of $n$ matchings of size $t$ in an $r$-uniform $r$-partite hypergraph $H$ that does not have a rainbow matching of size $t$. Let $t'=s+t=n-(q-1)t$, then $t'\geq t$ implies that $n\leq\floor{(t'/3r)^r}$, so by Theorem \ref{theo:psz}, there exists a collection of $n$ matchings of size $t'$ in an $r$-uniform $r$-partite hypergraph $H'$ that does not have a rainbow matching of size $t'$. Apply Lemma \ref{lemma:blowup} with $t_1=\ldots=t_{q-1}=t$ and $t_q=t'$, we get a collection of $n$ matchings of size $\sum_{i=1}^qt_i=n$ that does not contain a rainbow matching of size $\sum_{i=1}^qt_i-q+1=n-q+1$. It follows from the definition that $g(r,n)\leq n-q\leq n-\frac1{12r}n^{\frac{r-1}r}$.
\end{proof}

\subsection{Proof of Theorem \ref{theo:strong}}\label{sec_hrn}

Firstly, we prove the lower bound of Theorem \ref{theo:strong}.
This follows easily from Theorem \ref{theo:upper}.

\begin{theo}\label{theo:strongupper}
For every $r\geq 3$ and $n>6^r$, there exists a collection of $n$ matchings, each of size $n+\frac1{12r}n^{\frac{r-1}r}-1$, in an $r$-partite $r$-uniform hypergraph that does not admit a rainbow matching of size $n$.
\end{theo}

\begin{proof}[Proof of Theorem \ref{theo:strongupper}]
Let $m:=\frac{1}{12r}n^{\frac{r-1}{r}}-1$.
By Theorem \ref{theo:upper}, there exists a collection $\mathcal M=\{M_1,\dots,M_n\}$ of $n$ matchings of size $n$ in an $r$-uniform $r$-partite graph $\mathcal H$ which does not contain a rainbow matching of size $n-m$.  
Let $\mathcal H'$ be the $r$-uniform hypergraph obtained by adding a set $S$ of $rm$ new vertices to $\mathcal H$ and $m$ new edges $e_1,\dots,e_m$ which are pairwise vertex-disjoint and lie in $S$.
It is easy to see that $\mathcal H'$ is $r$-partite.
Consider the collection $\mathcal M'=\{M_1',\dots,M_n'\}$ of $n$ matchings of size $n+m$ in $\mathcal H'$ where $M_i'=M_i\cup\{e_1,\dots,e_m\}$ for every $i\in[n]$.
Suppose for a contradiction that $\mathcal M'$ yields a rainbow matching of size $n$.
By removing the edges $\{e_1,\dots,e_m\}$ from such matching, we obtain a matching of size at least $n-m$ which is rainbow with respect to $\mathcal M$, a contradiction.
Therefore, $\mathcal M'$ does not contain a rainbow matching of size $n$, which concludes the proof.
\end{proof}

To prove the upper bound in Theorem \ref{theo:strong}, we proceed in two steps.
We first use Lemma \ref{lem:gibounds} to prove a weak asymptotic version, namely, that every collection of $n$ matchings of size roughly $(r+1)n/2$ admits a rainbow matching of size $n-o(n)$.
We then employ the sampling trick of Munh\'{a}-Correia, Pokrovskiy and Sudakov to obtain the upper bound in Theorem \ref{theo:strong}.
We start by proving the weak asymptotic result.

\begin{lemma}\label{lem:stronglower}
For every $r\geq3$, every collection of $n$ matchings of size $\ceil{\frac{(r+1)n}2}$ in an $r$-uniform hypergraph contains a rainbow matching of size $n-2^r\sqrt n$. 
\end{lemma}
\begin{proof}
Let $\mathcal{H}=\{H_1, \ldots, H_n\}$ be a collection of $n$ matchings of size $\ceil{\frac{(r+1)n}2}$ in an $r$-uniform hypergraph. 
Let $M$ be a maximum rainbow matching for $\mathcal{H}$ which has size $m$. By Lemma \ref{lem:gibounds},
\[\frac{(r+1)(n-m)^2}{r-1}\leq(n-m)\frac{2\ceil{\frac{(r+1)n}2}-(r+1)m}{r-1}\leq\sum_{i>m} g_i \leq \frac12\binom{2r}{r}m.\]
After simplification, we get
\[(r+1)m^2-\left(2(r+1)n+\frac12(r-1)\binom{2r}{r}\right)m+(r+1)n^2\leq 0.\]
Solving this gives
\begin{align*}
m&\geq\frac{2(r+1)n+\frac12(r-1)\binom{2r}{r}-\sqrt{\left(2(r+1)n+\frac12(r-1)\binom{2r}{r}\right)^2-4(r+1)^2n^2}}{2(r+1)}\\
&=n-\frac{\sqrt{2(r^2-1)\binom{2r}{r}n+\frac14(r-1)^2\binom{2r}{r}^2}-\frac12(r-1)\binom{2r}{r}}{2(r+1)}\\
&=n-\frac{2(r^2-1)\binom{2r}{r}n}{2(r+1)\left(\sqrt{2(r^2-1)\binom{2r}{r}n+\frac14(r-1)^2\binom{2r}{r}^2}+\frac12(r-1)\binom{2r}{r}\right)}\\
&\geq n-\frac{(r-1)\binom{2r}{r}n}{\sqrt{2(r^2-1)\binom{2r}{r}n}}=n-\sqrt{\frac{(r-1)\binom{2r}r}{2(r+1)}}\sqrt n\geq n-2^r\sqrt n,
\end{align*}
completing the proof.
\end{proof}

For the sampling trick, we use the following version of Chernoff's bound. 
\begin{lemma}[Chernoff's bound, Corollary 2.3 in \cite{jlr}]\label{lem:chernoff}
Let $X\sim B(n,p)$ be a binomial random variable, then for every $0<\varepsilon<1$, 
\[\mathbb{P}\left(|X-\mathbb{E}X|\geq\varepsilon\mathbb{E}X\right)\leq2\exp\left(-\frac{\varepsilon^2}3\mathbb{E}X\right).\]
\end{lemma}

We are now able to apply the sampling trick of Munh\'{a}-Correia, Pokrovskiy and Sudakov to obtain the upper bound on $h'(r,n)$ given in Theorem \ref{theo:strong}.

\begin{theo}\label{thm:stronglower}
For every $r\geq3$ and $n$ sufficiently large, every collection of $n$ matchings of size at least $\frac{(r+1)n}2+3r^2n^{\frac{2r-1}{2r}}$ in an $r$-uniform hypergraph contains a rainbow matching of size $n$. 
\end{theo}
\begin{proof}[Proof of Theorem \ref{thm:stronglower}]
Let $M_1, \ldots, M_n$ be the $n$ matchings and let $H$ be the $r$-uniform hypergraph. Let $S\subset V(H)$ be a random set obtained by including every vertex independently with probability $p=4n^{-\frac1{2r}}$. For each $i\in[n]$, the expected number of edges in $M_i$ with all of their vertices in $S$ is 
\[p^r\left(\frac{(r+1)n}2+3r^2n^{\frac{2r-1}{2r}}\right)\geq r2^{r+1}\sqrt{n}.\]
Thus, by Lemma \ref{lem:chernoff} and a union bound, for all sufficiently large $n$, the probability that there exists $i\in[n]$ such that less than $r2^r\sqrt n$ edges of $M_i$ have all of their vertices in $S$ is at most 
\[2n\exp\left(-\frac1{12}r2^{r+1}\sqrt{n}\right)\leq\frac13.\]

Similarly, for every $i\in[n]$, the expected number of edges in $M_i$ with none of their vertices in $S$ is 
\[(1-p)^r\left(\frac{(r+1)n}2+3r^2n^{\frac{2r-1}{2r}}\right)\geq(1-pr)\left(\frac{(r+1)n}2+3r^2n^{\frac{2r-1}{2r}}\right)\geq\frac{(r+1)n}2+\frac {r^2}4n^{\frac{2r-1}{2r}},\]
provided that $n$ is large. Thus, by applying Lemma \ref{lem:chernoff} with $\varepsilon=n^{-\frac{1}{2r}}$ and using a union bound, we see that for all sufficiently large $n$, the probability that there exists $i\in[n]$ such that fewer than $\frac{(r+1)n}2$ edges of $M_i$ have none of their vertices in $S$ is at most 
\[2n\exp\left(-\frac16(r+1)n^{\frac{r-1}r}\right)\leq\frac13.\]

Therefore, there is a realisation of $S$ such that for each $i\in[n]$, there are at least $r2^r\sqrt n$ edges in $M_i$ with all of their vertices in $S$, and at least $\frac{(r+1)n}2$ edges in $M_i$ with none of their vertices in $S$. By Lemma \ref{lem:stronglower}, there is a rainbow matching of size at least $n-2^r\sqrt n$ in $H-S$, which can be greedily extended to a rainbow matching of size $n$ by using edges in $S$. 
\end{proof}

\section{Concluding Remarks} \label{sec_conc}

We finish this note by including some details on the best bounds of several related problems.

The two problems we discuss above ask for the largest rainbow matching guaranteed in every collection of $n$ matchings of size $n$, and the smallest size we require each of $n$ matchings to have in order to obtain a rainbow matching of size $n$. A third naturally related question to consider is how many matchings of size $n$ are needed to ensure the existence of a rainbow matching of size $n$.

Let $f(r,n)$ denote the smallest number of matchings of size $n$ in an $r$-uniform $r$-partite hypergraph that guarantees a rainbow matching of size $n$, and let $F(r,n)$ denote the smallest number of matchings of size $n$ in an $r$-uniform hypergraph that ensures a rainbow matching of size $n$. Note that clearly $f(r,n) \leq F(r,n)$. For the $r$-partite version of this problem, Aharoni and Berger~\cite[Conjecture 1.2]{ab} conjectured that $f(r,n)=2^{r-1}(n-1)+1$. (In fact their conjecture as originally stated is slightly more general but we focus on this case.) They proved that the conjecture holds when $r=2$, but noted that this was essentially already proved by Drisko~\cite{drisko}, just in the less general setting of Latin rectangles. Concerning $F(2,n)$, it was conjectured in \cite{abchs} that $F(2,n)=2n$ for even $n$ and $F(2,n)=2n-1$ when $n$ is odd, which would be tight due to extremal constructions. They also proved an upper bound of $3n-2$ which has since been improved to $3n-3$ in \cite{abkk}. Regarding $f(r,n)$ and $F(r,n)$ for $r \geq 3$, Alon~\cite{alon} disproved Aharoni and Berger's conjecture for certain values of $r$ and $n$, including almost all cases with $r \geq 4$, and gave a superexponential upper bound on $F(r,n)$ for all $r$. Glebov, Sudakov and Sz\'{a}bo~\cite{gss} improved this upper bound to a polynomial in $n$. Recently Pohoata, Sauermann and Zakharov~\cite{psz} determined the exact order of the polynomial, showing that for all $r\geq 3$ there exist positive constants $c_r$ and $C_r$ such that $c_rn^r \leq f(r,n) \leq F(r,n) \leq C_rn^r$. It is their Lemma 3.1 proving the lower bound here on $f(r,n)$ that we used in Section \ref{sec_ub}. Note that, in this paper, we focused on fixing $r$ and letting $n$ grow. One can also consider letting $r$ grow, or considering $n$ to be fixed and there are some further results known in these directions given in the listed literature.

Recalling that the Ryser-Brualdi-Stein conjecture concerns only edge-disjoint matchings, it is natural to consider making progress on these related problems by restricting back to the case where the union of the matchings forms a simple graph. In this setting, Chakraborti and Loh~\cite{cl} showed that $(2+o(1))n$ edge-disjoint matchings of size $n$ always admit a rainbow matching of size $n$, giving further evidence towards the conjecture concerning $F(2,n)$.

We finish this note by making a formal conjecture concerning the asymptotics of $g(r,n)$, $g'(r,n)$, $h(r,n)$ and $h'(r,n)$.

\begin{conj}
    Let $r\geq2$ be fixed. Then for each $x \in \{g(r,n), g'(r,n), h(r,n), h'(r,n)\}$,
    $$\lim_{n \ra \infty} \frac{x}{n} = 1.$$
\end{conj}

\section{Acknowledgements}
This work was initiated while the authors were visiting the University of Warwick as part of a workshop organised by Richard Montgomery and funded by ERC Starting Grant 947978 in August 2024.
We are very grateful to the University of Warwick and the organisers for the support to be able to attend the event and the great working environment.

\end{document}